\documentclass[12pt]{amsart}

\usepackage{amssymb}
\usepackage{bm}
\usepackage{graphicx}
\usepackage{psfrag}
\usepackage{hyperref}
\usepackage[usenames,dvipsnames]{xcolor}
\usepackage{enumerate}
\usepackage{url}

\usepackage{algpseudocode}

\usepackage{mathtools}

\usepackage{xy}
\input xy
\xyoption{all}

\newcommand{\excise}[1]{}

\newtheorem{thm}{Theorem}[section]
\newtheorem{lemma}[thm]{Lemma}

\newtheorem{cor}[thm]{Corollary}

\theoremstyle{definition}

\newtheorem{example}[thm]{Example}
\newtheorem{remark}[thm]{Remark}
\newtheorem{defn}[thm]{Definition}

\numberwithin{equation}{section}

\newcommand{\ring}[1]{\ensuremath{\mathbb{#1}}}

\renewcommand\>{\rangle}
\newcommand\<{\langle}

\newcommand\ZZ{\ring{Z}}

\def\ol#1{{\overline {#1}}}

\DeclareMathOperator\lcm{lcm} 

\begin{document}

\mbox{}
\title{Squarefree divisor complexes of certain numerical semigroup elements}

\author[J.~Autry]{Jackson Autry}
\address{Mathematics Department\\San Diego State University\\San Diego, CA 92182}
\email{j.connor.autry@gmail.com}

\author[P.~Graves]{Paige Graves}
\address{Mathematics Department\\University of La Verne\\La Verne, CA 91750}
\email{paige.graves@laverne.edu}

\author[J.~Loucks]{Jessie Loucks}
\address{Mathematics Department\\Sacramento State University\\Sacramento, CA 95819}
\email{jessieloucks13@gmail.com}

\author[C.~O'Neill]{Christopher O'Neill}
\address{Mathematics Department\\San Diego State University\\San Diego, CA 92182}
\email{coneill@math.ucdavis.edu}

\author[V.~Ponomarenko]{Vadim Ponomarenko}
\address{Mathematics Department\\San Diego State University\\San Diego, CA 92182}
\email{vponomarenko@sdsu.edu}

\author[S.~Yih]{Samuel Yih}
\address{Mathematics Department, Pomona College, 610 College Ave., Claremont, CA 91711} 
\email{sy012014@mymail.pomona.edu}




\date{\today}

\begin{abstract}
A numerical semigroup $S$ is an additive subsemigroup of the non-negative integers with finite complement, and the squarefree divisor complex of an element $m \in S$ is a simplicial complex $\Delta_m$ that arises in the study of multigraded Betti numbers.  We compute squarefree divisor complexes for certain classes numerical semigroups, and exhibit a new family of simplicial complexes that are occur as the squarefree divisor complex of some numerical semigroup element.
\end{abstract}

\maketitle

\section{Introduction}
\label{sec:intro}

A numerical semigroup is a cofinite subset $S \subseteq \ZZ_{\ge 0}$ that is closed under addition.  The numerical semigroup generated by $\{t_1,t_2,\ldots, t_k\} \subset \ZZ_{\ge 0}$ is the set
$$\<t_1, t_2, \ldots, t_k\> = \{\alpha_1t_1 + \cdots + \alpha_kt_k : \alpha_i \in \ZZ_{\ge 0}\}.$$
If $t \in S$ is not a sum of strictly smaller elements of $S$, we call $t$ irreducible in $S$.  The~set $\{n_1, n_2,\ldots, n_d\}$ of irreducible elements of $S$ is the unique minimal generating set of $S$, i.e.\ $S = \langle n_1, n_2, \ldots, n_d\rangle$.  Unless otherwise stated, for the remainder of the document, whenever we write $S = \<n_1, \ldots, n_d\>$, we assume $n_1, \ldots, n_d$ are precisely the irreducible elements of $S$.  In this case, we call $d$ the embedding dimension of $S$.  

Given numerical semigroup $S$, the Frobenius number of $S$, denoted $F(S)$, is the largest element of $\ZZ_{\ge 0}\setminus S$.  We call $x\in\ZZ_{\ge 0}\setminus S$ a pseudo-Frobenius number if $x+s\in S$ for all $s\in S\setminus \{0\}$, and denote by $PF(S)$ the set of all such $x$, noting in particular that $F(S)\in PF(S)$.  For a more thorough introduction to numerical semigroups, see \cite{MR2549780}.

Fix a numerical semigroup $S = \<n_1, n_2, \ldots, n_d\>$, and let $[d] = \{1, \ldots, d\}$.  A simplicial complex is a collection $\Delta \subset 2^{[d]}$ of subsets of $[d]$ (called faces) such that if $F \in \Delta$ and $F' \subset F$, then $F' \in \Delta$.  The maximal faces of $\Delta$ are called facets.  Given a face $F \subset [d]$ and $k \in \ZZ$, define $F^{\le k} = \{i \in F : i \le k\}$ and $F^{>k} = \{i \in F : i > k\}$.  The Euler characteristic of $\Delta$ is
$$\chi(\Delta) = \sum_{F \in \Delta}(-1)^{|F|}.$$
For each face $F \subset [d]$, write $n_F = \sum_{i \in F} n_i$.  Given $m \in S$, we define the squarefree divisor complex of $m$ in $S$ to be the simplicial complex
$$\Delta^S_m = \{F\in [d]: m - n_F \in S\}.$$
When there can be no confusion, we omit the superscript and write $\Delta_m$ instead of $\Delta_m^S$.  For a full treatment on simplicial complexes, see \cite[Section~2]{stanleybook}.  

Squarefree divisor complexes were first defined in~\cite{MR1481087} in the context of semigroup rings, for use in studying multigraded Betti numbers.  Specifically, they prove the Hilbert series of a numerical semigroup $S$ (see \cite{bookapplications}) can be expressed as
$$\mathcal H(S;t) = \sum_{m \in S} t^m = \frac{\sum_{m \in S} \chi(\Delta_m)t^m}{(1 - t^{n_1}) \cdots (1 - t^{n_d})}$$
Hilbert series have long been used to study numerical semigroups~\cite{continuousdiscrete}.  There has also been a recent surge of interest~\cite{frobvectors,shortgenfuncs,restrictedpartitionsyzygies,cyclotomicmonthly} yielding, among other things, characterizations of the numerator above in several cases~\cite{cyclotomic1,almostsymmetrichilbert,aperyhilbert}.  
The authors of~\cite{MR1481087} also pose the question ``which simplicial complexes $\Delta$ are realizable as the squarefree divisor complex~$\Delta_m^S$ for some $S$ and some $m \in S$'', and provided a complete answer in the case where $\Delta$ is a graph.  

Our contribution is twofold.  First, we provide a novel method of iteratively constructing squarefree divisor complexes and use it to produce a new family of realizable squarefree divisor complexes.  Second, for several classes of numerical semigroups, we explicitly compute the squarefree divisor complex of every element~$m$ for which~$\Delta_m$ has nonzero Euler characteristic (i.e.\ those appearing in the numerator of the Hilbert series equation above).  For each class of semigroups, the elements with nonzero Euler characteristic are known \cite{cyclotomic1}, but the original proof used high level machinery 
rather than squarefree divisor complexes.  Our work provides an alternative proof ``from first principles'' in these cases.

\section{Realization of simplicial complexes}
\label{sec:realization}

The main result of this section is Corollary~\ref{c:fatforest}, which states that any fat forest (Definition~\ref{d:fatforest}) can be realized as the squarefree divisor complex of some numerical semigroup element.  Our proof comes in two steps.  
First, we give a construction that allows us to produce disjoint unions of squarefree divisor complexes (Theorem~\ref{t:disjointunion}) using a gluing construction outlined in Lemma~\ref{l:gluing}.  This construction also appeared as \cite[Lemma~2.6(a)]{MR1481087}, though it was presented without proof and with insufficiently strong hypotheses (see Example~\ref{e:disjointunionrelativelyprime}).  Second, we iterate a construction called ``inflation'' to build any fat tree one vertex at a time.  

\begin{defn}\label{d:fatforest}
A simplicial complex $\Delta$ is a \emph{fat tree} if its facets $F_1, F_2, \ldots$ can be ordered in such a way that the intersection of each $F_j$ with $F_1 \cup \cdots \cup F_{j-1}$ is a simplex.  We say $\Delta$ is a \emph{fat forest} if it is a disjoint union of fat trees.  
\end{defn}

\begin{example}\label{e:fatforest}
Intuitively, any fat tree $\Delta$ can be constructed one facet at a time by attaching each new facet along an existing face.  In doing so, no cycles can be formed.  Figure~\ref{f:fattree} depicts one such complex, whose vertices are labeled in accordance with one possible facet order.  
\end{example}

\begin{figure}
\begin{center}
\includegraphics[width=3.2in]{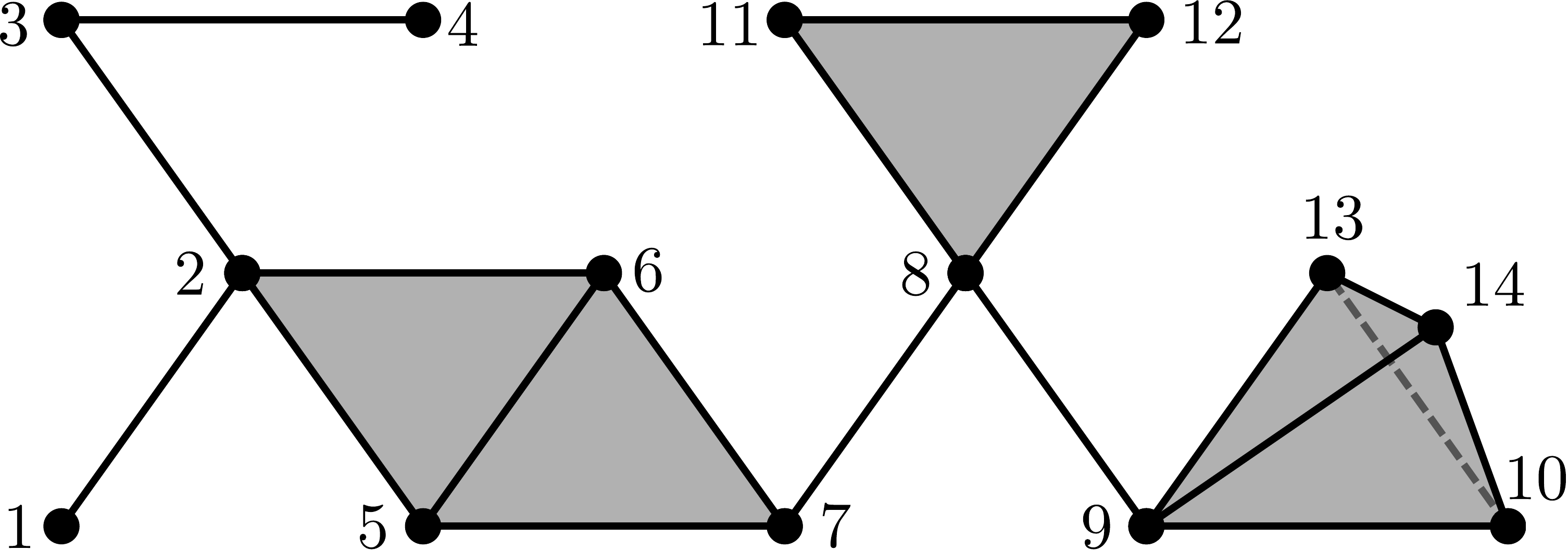}
\end{center}
\caption{An example of a fat tree.}
\label{f:fattree}
\end{figure}

\begin{lemma}[{\cite[Lemma~8.8]{MR2549780}}]\label{l:gluing}
Fix numerical semigroups $S = \<n_1, n_2, \ldots, n_d\>$, and $S' = \<n_1', n_2', \ldots, n_{d'}'\>$.  If $k \in S \setminus N$ and $k' \in S' \setminus N'$ with $\gcd(k,k') = 1$, then 
$$k'S + kS' = \<k'n_1, k'n_2, \ldots, k'n_d, kn_1', kn_2', \ldots , kn_{d'}'\>$$
has embedding dimension $d+d'$.  
\end{lemma}


\begin{thm}[Disjoint union]\label{t:disjointunion}
Fix numerical semigroups $S = \<n_1, n_2, \ldots, n_d\>$ and $S' = \<n_1', n_2', \ldots, n_{d'}'\>$, and let $T = k'S + kS'$.  For coprime $k \in S$ and $k' \in S'$,
$$\Delta_{kk'}^T = \Delta^S_k \cup (\Delta^{S'}_{k'} + d),$$
a disjoint union.
\end{thm}

\begin{proof}
First, note that the union is disjoint as $\Delta_k^S \subseteq 2^{[d]}$ while $\Delta_{k'}^{S'} + d \subseteq 2^{[d + d'] \setminus [d]}$.  
If $F \in \Delta^S_k$, then $k'k - k'n_F = k'(k - n_F) \in S$, so $kk' - n_F \in T$, meaning $F \in \Delta^T_{kk'}$.  By~symmetry, this implies $\Delta_{kk'}^T \supseteq \Delta^S_k \cup (\Delta^{S'}_{k'} + d)$.  

Conversely, fix $F\in \Delta^T_{kk'}$.  Suppose $F$ contains elements from both $[d]$ and $[d+d'] \setminus [d]$.  We can conclude $\{a, b\} \in \Delta^T_{kk'}$ for some $a \in [d]$ and $b \in [d + d'] \setminus [d]$.  This means $kk' - k'n_a - kn_{b-d}' \in T$, so we must have 
$$kk' = k'(n_a + m) + k(n_{b-d}' + m')$$
for some $m \in S$ and $m' \in S'$.  Taking the above equality modulo $k$, we see that $k \mid (n_a + m)$, and since $n_a \neq 0$, we must have $n_a + m \ge k$.  But now 
$$k'(n_a + m) + k(n_{b-d}' + m') \ge k'k + kn_{b-d}' > k'k,$$
which is a contradiction.  

Now, this means any $F \in \Delta^T_{kk'}$ must be entirely contained in $[d]$ or in $[d + d'] \setminus [d]$.  By~symmetry, it suffices to address the case $F \subseteq [d]$, wherein $k'k - k'n_F \in T$.  For some $m \in S$ and $m' \in S'$, we have $kk' - k'n_F = k'm + km'$.  Taking both sides modulo $k'$, we conclude $k' \mid m'$.  If $m' > 0$, then $m' \ge k'$, which yields a contradiction since then we would have $kk' \ge k'n_F + km' \ge k'n_F + kk'$.  Hence $m' = 0$ and dividing by $k'$ yields $k - n_{F} = m \in S$, so $F \in \Delta_k^S$.  By symmetry, we conclude $\Delta_{kk'}^T = \Delta^S_k \cup (\Delta^{S'}_{k'} + d)$.  
\end{proof}

\begin{example}\label{e:disjointunionrelativelyprime}
The ``coprime'' hypothesis in Theorem~\ref{t:disjointunion} can't be omitted.  For instance, if $k = 2 \in S = \<2,5\>$ and $k' = 6 \in S' = \<6, 10, 15\>$, then we have 
$T = k'S + kS' = \<12, 20, 30\>$ and $\Delta_2^S$, $\Delta_6^{S'}$ and $\Delta_{12}^T$ each have exactly one vertex.  
\end{example}


\begin{thm}\label{t:inflation}
Fix a numerical semigroup $S = \<n_1, \ldots, n_d\>$, an element $m \in S$, and a set $F \subset [d]$.  Let $b = m - n_F$, fix $p \in \ZZ_{\ge 0}$ with $\gcd(p,b) = 1$, and write $T = pS + b\<1\>$.  If $b \in S \setminus \{0\}$ and $\Delta_{n_F} = 2^F$, then the squarefree divisor complex $\Delta_{pm}^T$ is the complex $\Delta_m^S$ with one new vertex $b$, and the only facet containing $b$ is $F \cup \{b\}$.  
\end{thm}

\begin{proof}
By Lemma~\ref{l:gluing}, each $pn_i$ is a minimal generator of $T$, so $\Delta_m^S \subset \Delta_{pm}^T$, and since
$$pm = p(m - n_F) + pn_F = pb + \sum_{i \in F} pn_i,$$
we have $F \cup \{b\} \in \Delta_{pm}^T$.  Now, we claim any face of $\Delta_{pm}^T$ containing $b$ lies in $F \cup \{b\}$.  Indeed, any factorization of $pm$ involving $b$ can be written as 
$$pm = a_0b +  \sum_{i = 1}^d a_ipn_i = a_0(m - n_F) + \sum_{i = 1}^d a_ipn_i,$$
and taking both sides modulo $p$ implies $p \mid a_0$.  Writing $a_0 = kp$ for $k \ge 1$, dividing by $p$, and adding $n_F - m$ to both sides yields, we have
$$n_F = (k - 1)(m - n_F) + \sum_{i = 1}^d a_in_i.$$
Since $b = m - n_F \in S$, the above yields a factorization for $n_F$, but since $\Delta_{n_F} = 2^F$, this means $a_i = 0$ for each $i \notin F$.  This completes the proof.  
\end{proof}

\begin{thm}\label{t:fattree}
Any fat tree $\Delta$ occurs as the squarefree divisor complex of some numerical semigroup element.  
\end{thm}

\begin{proof}
If $\Delta$ consists only of a simplex, then it equals the squarefree divisor complex of $1 \in \<1\>$.  From here, we proceed by induction on the number of vertices of $\Delta$.  Let $F_1, F_2, \ldots, F_k$ denote the facets of $\Delta$, let $\Delta' = F_1 \cup \cdots \cup F_{k-1}$, and suppose $F_k \cap \Delta'$ is a simplex $F \in \Delta'$ with $|F_k| = |F| + 1$.  For our induction, we assume that $\Delta' = \Delta_m^S$ with $S = \<n_1, \ldots, n_d\>$, and that any sum of distinct generators is uniquely factorable (in particular, this implies $\Delta_{n_G} = 2^G$ for any $G \subset [d]$).  Fix any prime $p > b = m - n_F$.  Since $F \in \Delta_m^S$, we have $b \in S$, so by Theorem~\ref{t:inflation}, $\Delta = \Delta_{pm}^T$ with $T = pS + b\<1\>$.  By~assumption on $p$, any sum of generators of $T$ including $b$ equals $b$ modulo $p$, and thus is also uniquely factorable.  This completes the proof.  
\end{proof}

\begin{cor}\label{c:fatforest}
Any fat forest occurs as the squarefree divisor complex of some numerical semigroup element.  
\end{cor}

\begin{proof}
Theorem~\ref{t:fattree} implies any fat tree can be realized, and by choosing different primes in Theorem~\ref{t:inflation}, their disjoint union can be realized by Theorem~\ref{t:disjointunion}.  
\end{proof}




\section{Elements with nonzero Euler characteristic}
\label{sec:euler}

In this section, we turn our attention to particular classes of numerical semigroups.

\subsection{Supersymmetric numerical semigroups}

Suppose $t_1, t_2, \ldots, t_d \in\ZZ_{\ge 2}$ are pairwise coprime, and let $L = t_1 t_2 \cdots t_d$.  Numerical semigroups of the form $\<L/t_1, \ldots, L/t_d\>$ are called \emph{supersymmetric} \cite{denumerant}.  Equivalently, a numerical semigroup $\<n_1, n_2, \ldots, n_d\>$ is supersymmetric if $L/n_1, \ldots, L/n_d$ are pairwise coprime, where $L = \lcm(n_1, \ldots, n_d)$.


\begin{thm}\label{t:supersymmetric1}
Suppose $t_1, t_2, \ldots, t_d \in \ZZ_{\ge 2}$ are pairwise coprime, let $L = t_1 t_2 \cdots t_d$, and let $S = \<L/t_1, \ldots, L/t_d\>$.  If $k \in \ZZ_{\ge 0}$, then $\Delta_{kT}$ contains exactly those faces $F$ satisfying $|F| \le k$.  Further, $\chi(\Delta_{kL})=(-1)^k\binom{d-1}{k}$, which is zero precisely when $k \ge d$.
\end{thm}

\begin{proof}
For any face $F \subseteq [d]$ with $|F| \le k$, we have
$$kT = (k - |F|)L + |F|L = (k - |F|)t_1(L/t_1) + \sum_{i \in F} t_i(L/t_i) \in S,$$
meaning $F \in \Delta_{kL}$.  On the other hand, suppose $|F| \ge k + 1$ and $kL - n_F \in S$, so there is some factorization $kL = \sum_{i=1}^d a_i (L/t_i)$ with $a_i > 0$ for each $i \in F$.  For each~$i$, taking the above equation modulo $t_i$, we conclude $t_i \mid a_i$.  As such, 
$$kL = \sum_{i=1}^d a_i (L/t_i) \ge \sum_{i \in F} a_i (L/t_i) \ge \sum_{i \in F} t_i (L/t_i) \ge (k+1)L,$$
which is impossible, so we conclude $F \notin \Delta_{kL}$.  

The last statement follows from the observation that there are $\binom{d}{i}$ faces with size $i$ and the well-known identity 
$\sum_{i=0}^k (-1)^i \binom{d}{i} = (-1)^k\binom{d-1}{k};$
see, e.g.\ \cite[p.~165]{MR1397498}.
\end{proof}

Theorem~\ref{t:supersymmetric1} identifies certain elements of supersymmetric $S$ that have nonzero Euler characteristic.  In fact, these are the only such elements.

\begin{thm}\label{t:supersymmetric2}
Suppose $t_1, t_2, \ldots, t_d \in \ZZ_{\ge 2}$ are pairwise coprime, let $L = t_1 t_2 \cdots t_d$, and let $S = \<L/t_1, \ldots, L/t_d\>$.  If $m \in S$ and $L \nmid m$, then $\chi(\Delta_m) = 0$.
\end{thm}

\begin{proof}
Consider two factorizations $m = \sum a_i (L/t_i) = \sum b_i (L/t_i)$.  For each $i$, define $\ol a_i$ to be the unique integer in $[0,a_i)$ congruent to $a_i$ modulo $t_i$, and define $\ol b_i$ similarly.  We can then write 
$$m = qL + \textstyle\sum \ol a_i (L/t_i) = q'L + \textstyle\sum \ol b_i (L/t_i)$$
for some $q, q' \in \ZZ_{\ge 0}$.  For each $i$, we take the above equality modulo $t_i$ to conclude $\ol a_i \equiv \ol b_i \bmod t_i$.  However, $0 \le \ol a_i, \ol b_i < t_i$, so $\ol a_i = \ol b_i$ for all $i$.  Furthermore, $q = q'$.

Now, let $\ol F = \{i : \ol a_i > 0\}$, which must be nonempty since $L \nmid m$, and fix $j \in \ol F$.  By the above argument, every factorization $m = \sum a_i (L/t_i)$ has $a_j > 0$, meaning if $G \subset [d]$ with $j \notin G$, then $G \in \Delta_m$ if and only if $G \cup \{j\} \in \Delta_m$.  As such, we conclude $\chi(\Delta_m) = 0$ since $G$ and $G \cup \{j \}$ have opposite signs in the formula for $\chi(\Delta_m)$.  
\end{proof}



\subsection{Embedding dimension 3 numerical semigroups}

We conclude by classifying the squarefree divisor complexes with nonzero Euler characteristic in any minimally 3-generated numerical semigroup.  

\begin{thm}\label{t:pseudofrob}
Fix a numerical semigroup $S = \<n_1, \ldots, n_d\>$, and let $n = n_1 + \cdots + n_d$.  For any $m \in S$, we have $\Delta_m = 2^{[d]} \setminus \{[d]\}$ if and only if $m - n \in PF(S)$.
\end{thm}

\begin{proof}
First, suppose $m - n \in PF(S)$.  This means $m - n \notin S$, but $m - n_F \in S$ for each proper subset $F \subsetneq [d]$, so each such $F \in \Delta_m$.  Conversely, suppose $\Delta_m = 2^{[d]} \setminus \{[d]\}$.  Since $[d] \notin \Delta_m$, $m - n \notin S$, but given any positive $s \in S$, writing $s = n_j + s'$ for some $s' \in S$ and some $j$, and letting $F = [d] \setminus \{j\}$, we have
$$m - n + s = m - (n - n_j) + s' = m - n_F + s' \in S,$$
which implies $m - n \in PF(S)$.
\end{proof}



We combine the above result with the following known result to prove Theorem~\ref{t:3gen}.

\begin{thm}[{\cite[Example~7.23]{MR2549780}}]\label{t:disconnected3gen}
For a numerical semigroup $S = \<n_1, n_2, n_3\>$, let 
$$m_i = \min \{m \in \ZZ_{\ge 0} : n_i \mid m \text{ and } m \in \<n_j, n_k\>\}$$ for each $i \in \{1,2,3\}$, where $\{i,j,k\} = \{1,2,3\}$.  The elements $m_1, m_2, m_3 \in S$ are the only elements of $S$ with disconnected squarefree divisor complexes.
\end{thm}



\begin{thm}\label{t:3gen}
Fix a numerical semigroup $S = \<n_1, n_2, n_3\>$, and define $m_1. m_2, m_3$ as above.  If $m \in S$, then $\chi(\Delta_m) \neq 0$ if and only if $m \in B = \{0,m_1, m_2, m_3\} \cup PF(S)$.  
\end{thm}

\begin{proof}
Up to symmetry, the only simplicial complexes on $[3]$ with nonzero Euler characteristic are those with facet sets $\emptyset, \{1,2\}, \{1,2,3\}, \{12,3\}$, and $\{12,13,23\}$.  Of these, the first only occurs as $\Delta_0$, Theorem~\ref{t:pseudofrob} implies the last complex coincides with $\Delta_m$ precisely when $m \in PF(S)$, and the remaining complexes are disconnected, which can only occur for some $m_i$ by Theorem~\ref{t:disconnected3gen}.  The final claim follows from the fact that $|PF(S)| \le 2$ by \cite[Corollary~9.22]{MR2549780}.  
\end{proof}

\begin{remark}\label{r:addfrobenius}
Let $n_1, n_2 \in \ZZ_{\ge 2}$ with $\gcd(n_1, n_2) = 1$, let 
$$n_3 = n_1n_2 - n_1 - n_2 = F(\<n_1, n_2\>),$$
and let $S = \<n_1, n_2, n_3\>$.  This family of 3-generated numerical semigroups arises in constructing the Frobenius variety~\cite{frobvariety}, and its factorization properties are of interest~\cite{deltalarge}.  Using Theorem~\ref{t:3gen}, we compute 
$$B = \{0, n_2 + n_3, n_1 + n_3, 2n_3, n_1 + 2n_3, n_2 + 2n_3\}.$$
\end{remark}

\bibliographystyle{plain}
\bibliography{aglopy}

\end{document}